\documentclass[a4paper]{article}
\usepackage[utf8]{inputenc}

\usepackage[margin=1.3in]{geometry}

\usepackage{amsthm,amssymb,amsfonts,amsmath}
\usepackage{amssymb,amsfonts,amsmath}
\usepackage{xcolor}
\newtheorem{theorem}{Theorem}
\newtheorem{lemma}[theorem]{Lemma}
\newtheorem{corollary}[theorem]{Corollary}

\newtheorem{observation}[theorem]{Observation}
\newtheorem{definition}{Definition}

\newtheorem{conjecture}{Conjecture}

\usepackage{graphicx}
\graphicspath{{figures/}}

\usepackage{thm-restate}
\usepackage{url}
\usepackage{faktor}
\usepackage{tikz-cd}

\title{Enclosing Depth and other Depth Measures}

\author{Patrick Schnider\thanks{Department of Mathematical Sciences,
        University of Copenhagen, Denmark. {\tt ps@math.ku.dk}}}
        
\date{}

\begin{document}

\maketitle

\begin{abstract}
We study families of depth measures defined by natural sets of axioms.
We show that any such depth measure is a constant factor approximation of Tukey depth.
We further investigate the dimensions of depth regions, showing that the \emph{Cascade conjecture}, introduced by Kalai for Tverberg depth, holds for all depth measures which satisfy our most restrictive set of axioms, which includes Tukey depth.
Along the way, we introduce and study a new depth measure called \emph{enclosing depth}, which we believe to be of independent interest, and show its relation to a constant-fraction Radon theorem on certain two-colored point sets.
\end{abstract}

\section{Introduction}

Medians are an important tool in the statistical analysis and visualization of data.
Due to the fact that medians only depend on the order of the data points, and not their exact positions, they are very robust against outliers.
However, in many applications, data sets are multidimensional, and there is no clear order of the data set.
For this reason, various generalizations of medians to higher dimensions have been introduced and studied, see e.g.~\cite{aloupis, Liu, Mosler} for surveys.
Many of these generalized medians rely on a notion of \emph{depth} of a query point within a data set, a median then being a query point with the highest depth among all possible query points.
Several such depth measures have been introduced over time, most famously Tukey depth~\cite{tukey} (also called halfspace depth), simplicial depth \cite{LiuSimplicial}, or convex hull peeling depth (see, e.g.,~\cite{aloupis}).
In particular, just like the median, all of these depth measures only depend on the relative positions of the involved points.
More formally, let $S^{\mathbb{R}^d}$ denote the family of all finite sets of points in $\mathbb{R}^d$.
A depth measure is a function $\varrho: (S^{\mathbb{R}^d},\mathbb{R}^d)\rightarrow \mathbb{R}_{\geq 0}$ which assigns to each pair $(S,q)$ consisting of a finite set of data points $S$ and a query point $q$ a value, which describes how deep the query point $q$ lies within the data set $S$.
A depth measure $\varrho$ is called \emph{combinatorial} if it depends only on the \emph{order type} of $S\cup\{q\}$.
The order type of a point set $S=\{s_1,\ldots,s_n\}$ is a mapping that assigns to each ordered $(d+1)$-tuple of points the orientation of the spanned simplex.
Another way to view this is the following: consider all hyperplanes spanned by $d$ or more points of $S$. This defines an arrangement $A$ of hyperplanes, whose \emph{cells} are the connected components of $\mathbb{R}^d\setminus A$. Then all points in a cell have the same depth.

In this paper, we consider general classes of combinatorial depth measures, defined by a small set of axioms, and prove relations between them and concrete depth measures, such as \emph{Tukey depth} ($\text{TD}$) and \emph{Tverberg depth} ($\text{TvD}$).
Let us first briefly discuss these two depth measures.

\begin{definition}
Let $S$ be a finite point set in $\mathbb{R}^d$ and let $q$ be a query point. Then the Tukey depth of $q$ with respect to $S$, denoted by $\textbf{TD}(S,q)$, is the minimum number of points of $S$ in any closed half-space containing $q$.
\end{definition}

Tukey depth, also known as \emph{halfspace depth}, was independently introduced by Joseph L.~Hodges in 1955 \cite{hodges} and by John W.~Tukey in 1975 \cite{tukey} and has received significant attention since, both from a combinatrial as well as from an algorithmic perspective, see e.g.~Chapter 58 in \cite{Handbook} and the references therein.
Notably, the \emph{centerpoint theorem} states that for any point set $S\subset\mathbb{R}^d$, there exists a point $q\in\mathbb{R}^d$ for which $\text{TD}(S,q)\geq\frac{|S|}{d+1}$ \cite{CP}.

In order to define Tverberg depth, we need a preliminary definition: given a point set $S$ in $\mathbb{R}^d$, an \emph{$r$-partition} of $S$ is a partition of $S$ into $r$ pairwise disjoint subsets $S_1,\ldots,S_r\subset S$ with $\bigcap_{i=1}^r\text{conv}(S_i)\neq\emptyset$.
We call $\bigcap_{i=1}^r\text{conv}(S_i)$ the \emph{intersection} of the $r$-partition.

\begin{definition}
Let $S$ be a finite point set in $\mathbb{R}^d$ and let $q$ be a query point. Then the Tverberg depth of $q$ with respect to $S$, denoted by $\textbf{TvD}(S,q)$, is the maximum $r$ such that there is an $r$-partition of $S$ whose intersection contains $q$.
\end{definition}

Tverberg depth is named after Helge Tverberg who proved in 1966 that any set of $(d+1)(r-1)+1$ points in $\mathbb{R}^d$ allows an $r$-partition \cite{tverberg}.
In particular, this implies that there is a point $q$ with $\text{TvD}(S,q)\geq\frac{|S|}{d+1}$.
Just as for Tukey depth, there is an extensive body of work on Tverberg's theorem, see the survey \cite{BaranyTverberg} and the references therein.

In $\mathbb{R}^1$, both Tukey and Tverberg depth give a very natural depth measure: it counts the number of points of $S$ to the left and to the right of $q$ and then returns the minimum of the two numbers.
We call this measure the \emph{standard depth} in $\mathbb{R}^1$.
In particular, for all of them there is always a point $q\in\mathbb{R}^1$ for which we have $\varrho(S,q)\geq\frac{|S|}{2}$, that is, a median.

Another depth measure that is important in this paper is called enclosing depth.
We say that a point set $S$ of size $(d+1)k$ in $\mathbb{R}^d$ \emph{$k$-encloses} a point $q$ if $S$ can be partitioned into $d+1$ pairwise disjoint subsets $S_1,\ldots,S_{d+1}$, each of size $k$, in such a way that for every transversal $p_1\in S_1,\ldots, p_{d+1}\in S_{d+1}$, the point $q$ is in the convex hull of $p_1,\ldots,p_{d+1}$.
If the sizes of $S_1,\ldots, S_{d+1}$ are not specified, we just say that $S_1,\ldots,S_{d+1}$ \emph{enclose} $q$.
Intuitively, if $S\cup\{q\}$ is in general position, the points of $S$ are centered around the vertices of a simplex with $q$ in its interior.
More formally, as we will see in Section \ref{sec:axiom2}, the convex hulls of the $S_i$ are \emph{well-separated}, meaning that for each $I\subseteq \{1,\ldots,d+1\}$ the convex hulls of $\bigcup_{i\in I}S_i$ and $\bigcup_{i\notin I}S_i$ can be separated by a hyperplane through $q$.

\begin{definition}
Let $S$ be a finite point set in $\mathbb{R}^d$ and let $q$ be a query point. Then the enclosing depth of $q$ with respect to $S$, denoted by $\textbf{ED}(S,q)$, is the maximum $k$ such that there exists a subset of $S$ which $k$-encloses $q$.
\end{definition}

It is straightforward to see that enclosing depth also gives the standard depth in $\mathbb{R}^1$.
The centerpoint theorem \cite{CP} and Tverberg's theorem \cite{tverberg} show that both for Tukey as well as Tverberg depth, there are deep points in any dimension.
The question whether a depth measure enforces deep points is a central question in the study of depth measures.
It turns out that this also holds for enclosing depth.
In fact, Fabila-Monroy and Huemer \cite{Fabila} have shown that enclosing depth can be bounded from below by a constant fraction of Tukey depth.
We give a new proof of this fact, which gives a slightly better constant.
We will further show that all depth measures considered in this paper can be bounded from below by enclosing depth.
From this we get one of the main results of this paper: all depth measures that satisfy the axioms given later are a constant factor approximation of Tukey depth.

Another area of study in depth measures are \emph{depth regions}, also called depth contours.
For some depth measure $\varrho$ and $\alpha\in\mathbb{R}$, we define the \emph{$\alpha$-region} of a point set $S\subset\mathbb{R}^d$ as the set of all points in $\mathbb{R}^d$ that have depth at least $\alpha$ with respect to $S$.
We denote the $\alpha$-region of $S$ by $D_{\varrho}^S(\alpha):=\{q\in\mathbb{R}^d\mid \varrho(S,q)\geq\alpha\}$.
Note that for $\alpha<\beta$ we have $D_{\varrho}^S(\alpha)\supset D_{\varrho}^S(\beta)$, that is, the depth regions are nested.
The structure of depth regions has been studied for several depth measures, see e.g.~\cite{Miller, Zuo}.
For example, it is well-known that Tukey depth regions are compact and convex.
Depth regions in $\mathbb{R}^2$ have been proposed as a tool for data visualization \cite{tukey}.
From a combinatorial point of view, Gil Kalai introduced the following conjecture \cite{Kalai_cascade}.

\begin{conjecture}[Cascade Conjecture]
Let $S$ be a point set of size $n$ in $\mathbb{R}^d$.
For each $i\in\{1,\ldots,n\}$, denote by $t_i$ the dimension of $D_{\text{TvD}}^S(i)$, where we set $\dim(\emptyset)=-1$.
Then
$$\sum_{i=1}^n t_i\geq 0.$$
\end{conjecture}

Here, the dimension of a subset $X$ of $\mathbb{R}^d$ is the maximum dimension of any neighborhood of a point in $X$.
The conjecture is known to be true when $S$ is in so-called \emph{strongly} general position \cite{Reay}, for general position in some dimensions \cite{Roudneff1, Roudneff2, Roudneff3} (see also \cite{BaranyTverberg} for more information), and without any assumption of general position for $d\leq 2$ in an unpublished M. Sc thesis in Hebrew by Akiva Kadari (see \cite{KalaiBirthday}).

While Kalai's conjecture is specifically about Tverberg depth, the sum of dimensions of depth regions can be computed for any depth measure, and thus the conjecture can be generalized to other depth measures.
In fact, in a talk Kalai conjectured that the Cascade conjecture is true for Tukey depth, mentioning on his slides that `this should be doable' \cite{KalaiTalk}.
In this work, we will prove the conjecture to be true for a family of depth measures that includes Tukey depth.

\subsection*{Structure of the paper}
We start the technical part by introducing a first set of axioms in Section \ref{sec:axiom1}, defining what we call \emph{super-additive} depth measures.
We show that these depth measures lie between Tukey depth and Tverberg depth.
In Section \ref{sec:cascade} we prove the cascade conjecture for super-additive depth measures whose depth regions are compact and convex.
We then give a second set of axioms in Section \ref{sec:axiom2}, defining \emph{central} depth measures, and show how to bound them from below by enclosing depth.
Finally, in Section \ref{sec:enclosing}, we give a new proof of a lower bound for enclosing depth in terms of Tukey depth.
In our proof, we notice a close relationship of enclosing depth with a version of Radon's theorem on certain two-colored point sets.

\section{A first set of axioms}
\label{sec:axiom1}

The first set of depth measures that we consider are \emph{super-additive} depth measures\footnote{We name both our families of depth measures after one of the conditions they satisfy. The reason for this is that the condition they are named after is the condition which separates this family from the other one.}.
A combinatorial depth measure $\varrho: (S^{\mathbb{R}^d},\mathbb{R}^d)\rightarrow \mathbb{R}_{\geq 0}$ is called super-additive if it satisfies the following conditions:
\begin{enumerate}
\item[(i)] for all $S\in S^{\mathbb{R}^d}$ and $q,p\in\mathbb{R}^d$ we have $|\varrho(S,q)-\varrho(S\cup\{p\},q)|\leq 1$ (sensitivity),
\item[(ii)] for all $S\in S^{\mathbb{R}^d}$ and $q\in\mathbb{R}^d$ we have $\varrho(S,q)=0$ for $q\not\in\text{conv}(S)$ (locality),
\item[(iii)] for all $S\in S^{\mathbb{R}^d}$ and $q\in\mathbb{R}^d$ we have $\varrho(S,q)\geq 1$ for $q\in\text{conv}(S)$ (non-triviality),
\item[(iv)] for any disjoint subsets $S_1,S_2\subseteq S$ and $q\in\mathbb{R}^d$ we have $\varrho(S,q)\geq\varrho(S_1,q)+\varrho(S_2,q)$ (super-additivity).
\end{enumerate}

It is not hard to show that a one-dimensional depth measure which satisfies these conditions has to be the standard depth measure (in fact, the arguments are generalized to higher dimensions in the following two observations) and that no three conditions suffice for this.
Further, it can be shown that both Tukey depth and Tverberg depth are super-additive.

We first note that the first two axioms suffice to give an upper bound:

\begin{observation}
\label{obs:upper_bound}
For every depth measure $\varrho$ satisfying (i) sensitivity and (ii) locality and for all $S\in S^{\mathbb{R}^d}$ and $q\in\mathbb{R}^d$ we have $\varrho(S,q)\leq \text{TD}(S,q)$.
\end{observation}

\begin{proof}
By the definition of Tukey depth, $\text{TD}(S,q)=k$ implies that we can remove a subset $S'$ of $k$ points from $S$ so that $q$ is not in the convex hull of $S\setminus S'$.
In particular, $\varrho(S\setminus S',q)=0$ by locality.
By sensitivity we further have $\varrho(S\setminus S',q)\geq\varrho(S,q)-k$, which implies the claim.
\end{proof}

Further, the last two axioms can be used to give a lower bound:

\begin{observation}
For every depth measure $\varrho$ satisfying (iii) non-triviality and (iv) super-additivity and for all $S\in S^{\mathbb{R}^d}$ and $q\in\mathbb{R}^d$ we have $\varrho(S,q)\geq \text{TvD}(S,q)$.
\end{observation}

\begin{proof}
Let $\text{TvD}(S,q)=k$ and consider a $k$-partition $S_1,\ldots,S_k$ with $q$ in its intersection.
By non-triviality we have $\varrho(S_i,q)\geq 1$ for each $S_i$.
Using super-additivity and induction we conclude that $\varrho(\bigcup_{i=1}^k S_i,q)\geq\sum_{i=1}^k\varrho(S_i,k)\geq k$.
\end{proof}

Finally, it is not too hard to show that $\text{TvD}(S,q)\geq\frac{1}{d}\text{TD}(S,q)$: consider a simplex $\Delta$ spanned by points in $S$ which contains $q$. Such a simplex exists by Carath\'{e}odory's theorem. As $\Delta$ has at most $d+1$ vertices and $q$ lies in its relative interior, any halfspace with $q$ on its boundary contains at most $d$ of the vertices of $\Delta$. Thus, removing $\Delta$ the Tukey depth decreases by at most $d$. While the Tukey depth is not 0, we can always find another simplex containing $q$, and removing it again decreases the Tukey depth by at most $d$, thus there are indeed at least $\frac{1}{d}\text{TD}(S,q)$ vertex-disjoint simplices containing $q$.
This argument appears in several papers about algorithmic aspects of Tverberg points, see e.g.\ \cite{Sariel} or \cite{PabloAlgo}, Lemma 2.2.
Combining these observations, we thus get the following.

\begin{corollary}
\label{cor:additive}
Let $\varrho$ be a super-additive depth measure. Then for every point set $S$ and query point $q$ in $\mathbb{R}^d$ we have
$$\text{TD}(S,q)\geq \varrho(S,q)\geq \text{TvD}(S,q) \geq \frac{1}{d}\text{TD}(S,q).$$
\end{corollary}

Let us note here that it could be that the factor $\frac{1}{d}$ in the last inequality could be improved.
Indeed, in the plane, we have that $\text{TvD}=\min\{\text{TD}, \lceil\frac{|S|}{3}\rceil\}$ \cite{Reay}.
This fails already in dimension 3 \cite{Avis}.
It would be interesting to see how much the factor $\frac{1}{d}$ can be improved.

From Corollary \ref{cor:additive} it follows that for any super-additive depth measure and any point set there is always a point of depth at least $\frac{|S|}{d+1}$, for example any Tverberg point.
On the other hand, there are depth measures which give the standard depth in $\mathbb{R}^1$ which are not super-additive, for example convex hull peeling depth or enclosing depth.

\begin{observation}
Enclosing depth satisfies conditions (i)-(iii), but not the super-additivity condition (iv)
\end{observation}

\begin{proof}
It follows straight from the definition that enclosing depth satisfies the conditions (i)-(iii).
To see that the super-additivity condition is not satisfied, consider the example in Figure \ref{fig:encl_additivity}.
The point $q$ has enclosing depth 1 with respect to both the set of blue points and the set of red points.
However, it can be seen that the enclosing depth of $q$ with respect to both the red and the blue points is still 1.
\end{proof}

\begin{figure}
\centering
\includegraphics[scale=1]{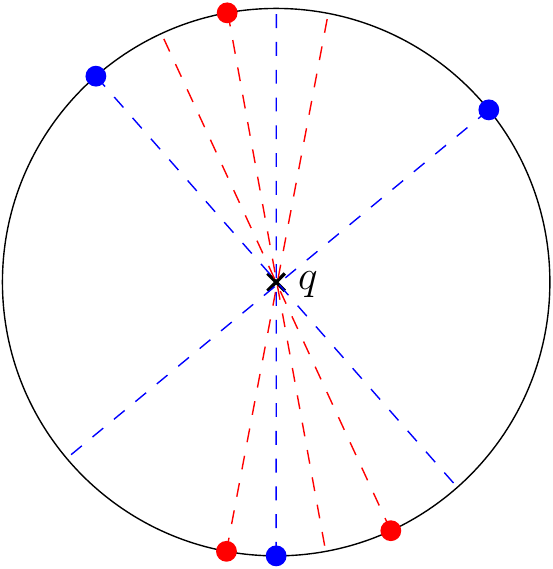}
\caption{Enclosing depth does not satisfy the super-additivity condition: the point $q$ has enclosing depth 1 with respect to both the blue and the red points, but its enclosing depth with respect to the union of the two sets is still 1.}
\label{fig:encl_additivity}
\end{figure}

\section{The Cascade Conjecture}
\label{sec:cascade}

In this section we prove the cascade conjecture for super-additive depth measures whose depth regions are compact and convex.
In fact, we will prove the cascade conjecture for the case of \emph{weighted point sets}.
This is not only to achieve greater generality, our proof of Lemma \ref{lem:cascade_partition} relies on the existence of weights, and does not go through without them.

A weighted point set is a finite point set $S$ together with a weight function $w:S\rightarrow\mathbb{R}_{\geq 0}$ which assigns a weight $w(p)$ to each $p\in S$.
We say that a weighted point set $S'$ is a strict subset of $S$, denoted by $S'\subset S$, if the underlying point set of $S'$ is a strict subset of the underlying point set of $S$, and $w'(p)\leq w(p)$ for every $p\in S'$, where $w'$ is the weight function on $S'$.
In particular, if $S'\subset S$, there is a point which is in $S$ but not in $S'$.
For two weighted point sets $A$ and $B$ with weight functions $w_A$ and $w_B$, respectively, the weight function on their union $A\cup B$ is defined as the sum of the respective weight functions.
That is, we have $w(p)=w_A(p)$ for $p\in A\setminus B$, $w(p)=w_B(p)$ for $p\in B\setminus A$ and $w(p)=w_A(p)+w_B(p)$ for $p\in A\cap B$.
Further, for a set $S$ of points we define the weight of $S$ as $w(S):=\sum_{p\in S}w(p)$.
Similarly, by a partition of a weighted point set $S$ into parts $A$ and $B$ we mean two weight functions $w_A$ and $w_B$, such that $w(p)=w_A(p)+w_B(p)$ for $p\in S$, and by a partition into strict subsets $A$ and $B$, we mean that both weighted point sets $A$ and $B$ must be strict subsets of $S$, that is, there are points $p_A,p_B$ in $S$ for which $w_A(p_A)=0$ and $w_B(p_B)=0$.
The axioms for super-additive depth measures extend to weighted point sets in the following way:

\begin{enumerate}
\item[(i)] for all $S\in S^{\mathbb{R}^d}$ and $q,p\in\mathbb{R}^d$ we have $|\varrho(S,q)-\varrho(S\cup\{p\},q)|\leq w(p)$ (sensitivity),
\item[(ii)] for all $S\in S^{\mathbb{R}^d}$ and $q\in\mathbb{R}^d$ we have $\varrho(S,q)=0$ for $q\not\in\text{conv}(S)$ (locality),
\item[(iii)] for all $S\in S^{\mathbb{R}^d}$ and $q\in\mathbb{R}^d$ we have $\varrho(S,q)\geq \min\{w(p):p\in S\}$ for $q\in\text{conv}(S)$ (non-triviality),
\item[(iv)] for any disjoint subsets $S_1,S_2\subseteq S$ and $q\in\mathbb{R}^d$ we have $\varrho(S,q)\geq\varrho(S_1,q)+\varrho(S_2,q)$ (super-additivity).
\end{enumerate}

Clearly, each point set can be considered as a weighted point set by assigning weight 1 to each point.
On the other hand, by placing several points at the same location, normalizing and using the fact that $\mathbb{Q}$ is dense in $\mathbb{R}$, each depth measure defined on point sets can be extended to weighted point sets.
Further, we can again define depth regions $D_{\varrho}^S(\alpha):=\{q\in\mathbb{R}^d\mid \varrho(S,q)\geq\alpha\}$.
We denote by $t_\alpha$ the dimension of $D_{\varrho}^S(\alpha)$.
We will also use a special depth region, called the \emph{median region}, denoted by $M_{\varrho}(S)$, which is the deepest non-empty depth region.
More formally, let $\alpha_0$ be the supremum value for which $D_{\varrho}^S(\alpha_0)\neq\emptyset$.
Then $M_{\varrho}(S):=D_{\varrho}^S(\alpha_0)$.

We now introduce a continuous version of the cascade condition:
\begin{definition}
Let $\varrho$ be a depth measure.
If for each weighted point set $S$ we have that
$$\int_0^{w(S)}t_\alpha \mathrm{d}\alpha\geq 0,$$
we say that $\varrho$ is \emph{cascading}.
\end{definition}

We say that a depth measure $\varrho$ is \emph{integral} if for any unweighted point set and any query point $q$ the depth $\varrho(S,q)$ is an integer.
Note that all the depth measures introduced so far are integral.
\begin{observation}
Let $\varrho$ be an integral depth measure and let $S$ be an unweighted point set of size $n$.
Then
$$\int_0^{w(S)}t_\alpha \mathrm{d}\alpha= \sum_{i=1}^n t_i$$
\end{observation}

\begin{proof}
As $\varrho$ is integral, we have that $t_\alpha=t_{\lfloor \alpha \rfloor}$.
In particular, for every integer $i$ we have
\[\int_{i-1}^{i}t_\alpha \mathrm{d}\alpha=t_i.\]
\end{proof}

In the following, we will show that super-additive depth measures whose depth regions are compact and convex are cascading in two steps.
First we will show that if we partition a weighted point set into two parts whose median regions intersect and the cascade condition holds for both parts, then the cascade condition holds for the whole set.
In a second step, we prove that we can always partition a point set in such a way, further enforcing that none of the parts contains all points, that is, each part is a strict subset.
The claim then follows by induction.

Before we do this, let us describe a way to compute $\int_0^{w(S)}t_\alpha \mathrm{d}\alpha$.
Consider some depth region $D_\varrho^S(\alpha)$ of dimension $k$ and assume without loss of generality that the origin lies in the median region.
Being convex, this depth region lies in some $k$-dimensional linear subspace $H\subset\mathbb{R}^d$.
Considering all depth regions, they lie in a sequence of nested linear subspaces, also known as a \emph{flag}.
We can find a basis $F=\{f_1,\ldots,f_d\}$ of $\mathbb{R}^d$ such that each relevant linear subspace is spanned by a subset of the basis vectors.
We call $F$ a \emph{basis} of $S$.

For each weighted point set $S$ and each vector $v$, we define the \emph{survival time} $\tau_S(v):=\sup\{\alpha\mid v\in \text{span}(D_\varrho^S(\alpha))\}$.
Similarly, we define $\tau_S(0):=\alpha_0$, where as above $\alpha_0$ is the supremum value for which $D_{\varrho}^S(\alpha_0)\neq\emptyset$.
In other words, we view $\tau_S(0)$ as the survival time of the origin.

\begin{lemma}
\label{lem:integral}
Let $F=\{f_1,\ldots,f_d\}$ be a basis of $\mathbb{R}^d$, and write $f_0:=0$. Then
\[ \int_0^{w(S)}t_\alpha d\alpha\geq\sum_{i=0}^d\tau_S(f_i)-w(S).\]
Further, if $F$ is a basis of $S$, then we have equality.
\end{lemma}

\begin{proof}
For an illustration of the proof see Figure \ref{fig:integral}.
Without loss of generality let $\tau_S(f_d)\leq\ldots\leq\tau_S(f_1)\leq\tau_S(f_0)$.
Consider some $\alpha\leq\tau_S(f_i)$.
The span of $D_\varrho^S(\alpha)$ contains all basis vectors $f_i,f_{i-1},\ldots,f_1$, thus we have $t_\alpha\geq i$.
We thus get
\begin{multline}
\int_0^{w(S)}t_\alpha d\alpha =
\int_0^{\tau_S(f_d)}t_\alpha d\alpha + \int_{\tau_S(f_d)}^{\tau_S(f_{d-1})}t_\alpha d\alpha +\ldots + \int_{\tau_S(f_1)}^{\tau_S(f_{0})}t_\alpha d\alpha + \int_{\tau_S(f_0)}^{w(S)}t_\alpha d\alpha \\
\geq \tau_S(f_d)\cdot d + (\tau_S(f_{d-1})-\tau_S(f_d))\cdot (d-1) +\ldots + (\tau_S(f_{0})-\tau_S(f_1))\cdot 0 + (w(S)-\tau_S(f_0))\cdot (-1) \\
= \tau_S(f_d)(d-(d-1)) + \tau_S(f_{d-1})((d-1)-(d-2)) +\ldots + \tau_S(f_0)(1-0)-w(S) =\sum_{i=0}^d\tau_S(f_i)-w(S).
\end{multline}

If $F$ is a basis of $S$, then for $\tau_S(f_{i+1})\leq\alpha\leq\tau_S(f_i)$ we have $t_\alpha= i$ and thus the above inequality is an equality.
\end{proof}


\begin{figure}
\centering
\includegraphics[scale=0.7]{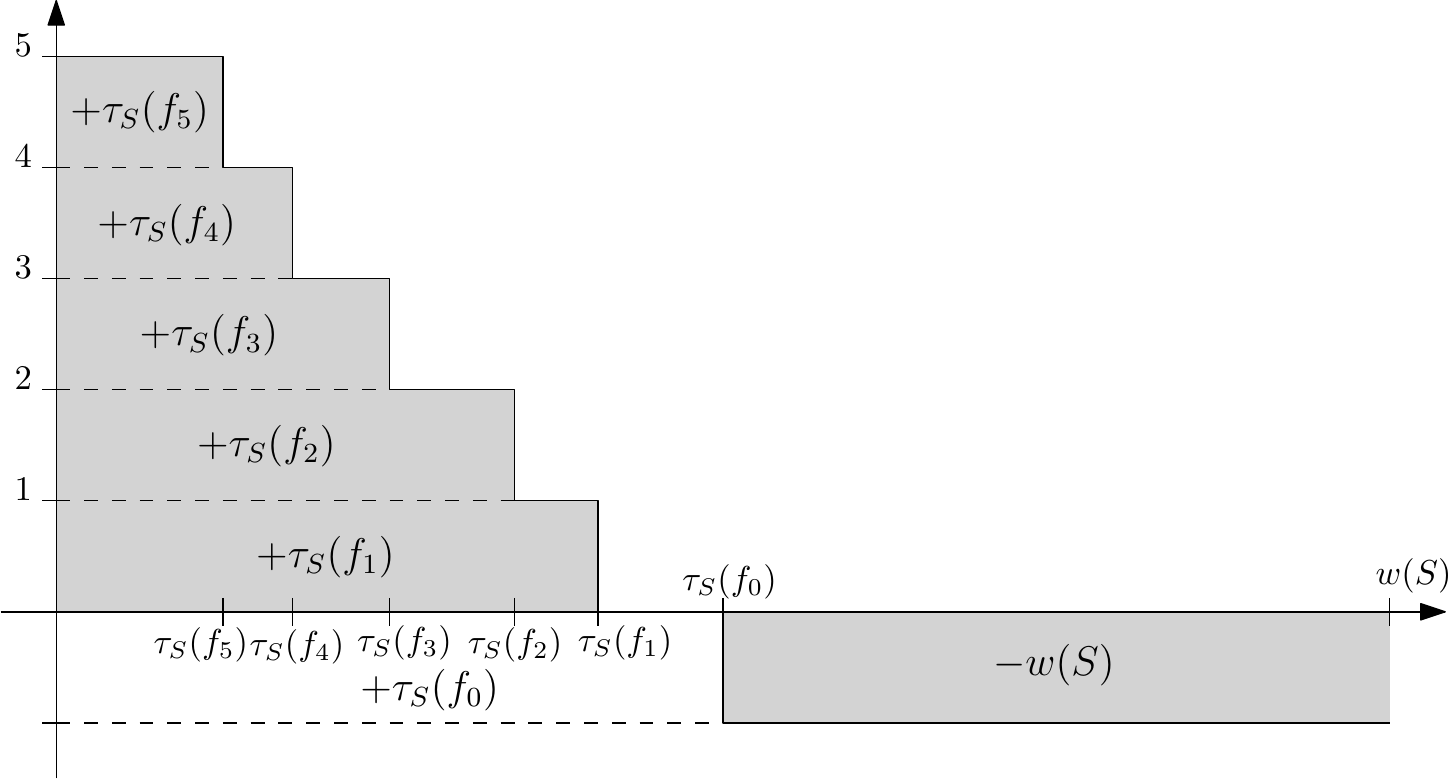}
\caption{$\int_0^{w(S)}t_\alpha d\alpha\geq\sum_{i=0}^d\tau_S(f_i)-w(S)$.}
\label{fig:integral}
\end{figure}

We will also use the following auxiliary lemma, which is well-known, see e.g.\ \cite{flags1}, Prop. 5.37 for a reference\footnote{In the reference, the result is stated for flags in $\mathbb{F}_q^n$, but as mentioned in a remark afterwards, the proof works for any vector space.}.

\begin{lemma}
\label{lem:common_basis}
Let $A$ and $B$ be two flags in a $d$-dimensional vector space $V$.
Then we can find a basis $F$ which is a common basis of both flags.
\end{lemma}


\begin{lemma}
\label{lem:cascade_sum}
Let $\varrho$ be a super-additive depth measure whose depth regions are convex and let $S_1$ and $S_2$ be two weighted point sets in $\mathbb{R}^d$ whose median regions intersect.
Assume that the cascade condition holds for $S_1$ and $S_2$.
Then the cascade condition holds for $S_1\cup S_2$.
\end{lemma}

\begin{proof}
We may assume without loss of generality that the origin is in both median regions.
By Lemma \ref{lem:common_basis}, we can choose a basis $F=\{f_1,\ldots,f_d\}$ of $\mathbb{R}^d$ that is a basis of both $S_1$ and $S_2$.
Again, we write $f_0=0$.

For $i\neq 0$, let $\ell_i$ be the line $\{\lambda f_i\mid \lambda\in\mathbb{R}\}$.
Note that by the convexity of the depth regions $f_i\in\text{span}(D_\varrho^S(\alpha))\}$ implies that $\ell_i\cap D_\varrho^S(\alpha)$ is an interval with non-empty relative interior.
In particular, assuming $f_i\in\text{span}(D_\varrho^{S_1}(\alpha_1))\}$ and $f_i\in\text{span}(D_\varrho^{S_2}(\alpha_2))\}$, we get two such intervals whose intersection has a non-empty relative interior.
Thus, there is a point $p\neq 0$ on $\ell_i$ which lies both in $D_\varrho^{S_1}(\alpha_1)$ and $D_\varrho^{S_2}(\alpha_2)$.
By the super-additivity condition, $p$ lies in $D_\varrho^{S_1\cup S_2}(\alpha_1+\alpha_2)$.
It follows that $\tau_{S_1\cup S_2}(f_i)\geq\tau_{S_1}(f_i)+\tau_{S_2}(f_i)$ for $i\neq 0$.
The same argument for the origin gives $\tau_{S_1\cup S_2}(f_0)\geq\tau_{S_1}(f_0)+\tau_{S_2}(f_0)$.
Thus, using Lemma \ref{lem:integral} we get
\begin{multline}
\int_0^{w(S_1\cup S_2)}t_\alpha d\alpha\geq
\sum_{i=0}^d\tau_{S_1\cup S_2}(f_i)-w(S_1\cup S_2)\geq\sum_{i=0}^d(\tau_{S_1}(f_i)+\tau_{S_2}(f_i))-(w(S_1)+w(S_2))\\
=\sum_{i=0}^d\tau_{S_1}(f_i)-w(S_1) + \sum_{i=0}^d\tau_{S_2}(f_i)-w(S_2)
=\int_0^{w(S_1)}t_\alpha d\alpha+\int_0^{w(S_2)}t_\alpha d\alpha\geq 0.
\end{multline}
\end{proof}

\begin{lemma}
\label{lem:cascade_partition}
Let $\varrho$ be a super-additive depth measure whose depth regions are compact and convex and let $S$ be a weighted point set in $\mathbb{R}^d$ with $|S|\geq d+2$.
Then there exists a partition of $S$ into strict subsets $S_1$ and $S_2$ whose median regions intersect.
\end{lemma}

\begin{proof}
Consider the barycentric subdivision $B$ of the boundary $\partial\Delta$ of the simplex with vertices $S$.
There is a natural identification of the vertices of $B$ with strict subsets of $S$ (see Figure \ref{fig:barycentric}).
Thus, for any such vertex $b\in B$ we get a strict weighted subset $S(b)$ with $w_{S(b)}(p)=w(p)$ if $p$ is in the subset $S(b)$ and $w_{S(b)}(p)=0$ otherwise.
We can extend this assignment linearly to $\partial\Delta$: for each point $b$ in the relative interior of a simplex $(b_1,\ldots,b_k)$ of $B$, consider its barycentric coordinates $(x_1,\ldots,x_k)$, $\sum x_i=1$, and define $w_{S(b)}(p):=\sum x_i\cdot w_{S(b_i)}(p)$ for each $p\in S$.
This defines a continuous map which assigns to each point $b$ on $\partial\Delta$ a strict weighted subset $S(b)$ of $S$.
Further, under the natural antipodality on $\partial\Delta$, we get complements of the weighted subsets, that is, $S(-b)=S(b)^C$.

We claim that for some point $b$ on $\partial\Delta$ we have that the median regions $M(b)$ and $M(-b)$ of $S(b)$ and $S(-b)$ intersect.
If this is true, our claim follows by setting $S_1=S(b)$ and $S_2=S(-b)$.
For a set $A$, denote by $-A$ its reflection at the origin.
For each $b$ on $\partial\Delta$, let $Q(b):=M(b)-M(-b)$ be the Minkowski sum of $M(b)$ and $-M(-b)$.
Note that $M(b)$ and $M(-b)$ intersect if and only if $0\in Q(b)$ and that $Q(-b)=-Q(b)$.
Further, as $Q(b)$ is a Minkowski sum of compact convex sets, $Q(b)$ is itself compact and convex.

Consider now the real vector bundle $\pi: E\rightarrow B$ obtained from attaching $\mathbb{R}^d$ to each point of $\partial\Delta$ and taking the quotient with respect to the antipodality and let $z$ be its zero section.
Following \cite{Zivaljevic}, we say that $\phi$ is a multivalued section if for every $x\in B$ we have that $\phi(x)\subseteq \pi^{-1}(x)$.
A multivalued section is said to be convex if $\phi(x)$ is convex for every $x\in B$.
A multivalued section is called compact if $\Gamma(\phi):=\{(x,v)\mid v\in\phi (x)\}$ is a compact set in $B\times E$.
It follows from the above arguments that $Q$ is a convex multivalued section.
Further, as each $Q(b)$ is compact and $\partial\Delta$ is compact, $Q$ is also compact.
Now, Proposition 1 from \cite{Zivaljevic} states that if $\pi: E\rightarrow B$ is a real vector bundle over a compact space $B$ which does not admit a nowhere zero section, then for every multivalued convex compact section $\phi$ there exists a point $x\in B$ for which $z(x)\in \phi(x)$.
Thus, in order to show that some $Q(b)$ contains the origin, it is sufficient to show that $\pi: E\rightarrow B$ does not admit a nowhere zero section.

Now, $\partial\Delta$ is homeomorphic to the sphere $S^{|S|-2}$, and the antipodality on $\partial\Delta$ corresponds to the standard antipodality on the sphere.
As $|S|\geq d+2$, the non-existence of a nowhere zero section thus follows from the Borsuk-Ulam theorem.
\end{proof}

While we have only shown that there is a partition, Bourgin-Yang-type theorems \cite{Bourgin, Yang} tell us, that the space of possible partitions has to be large.
In particular, it has dimension at least $|S|-d-2$.
Depending on the application, this might be used to enforce other conditions on the partitions.

\begin{figure}
\centering
\includegraphics[scale=0.7]{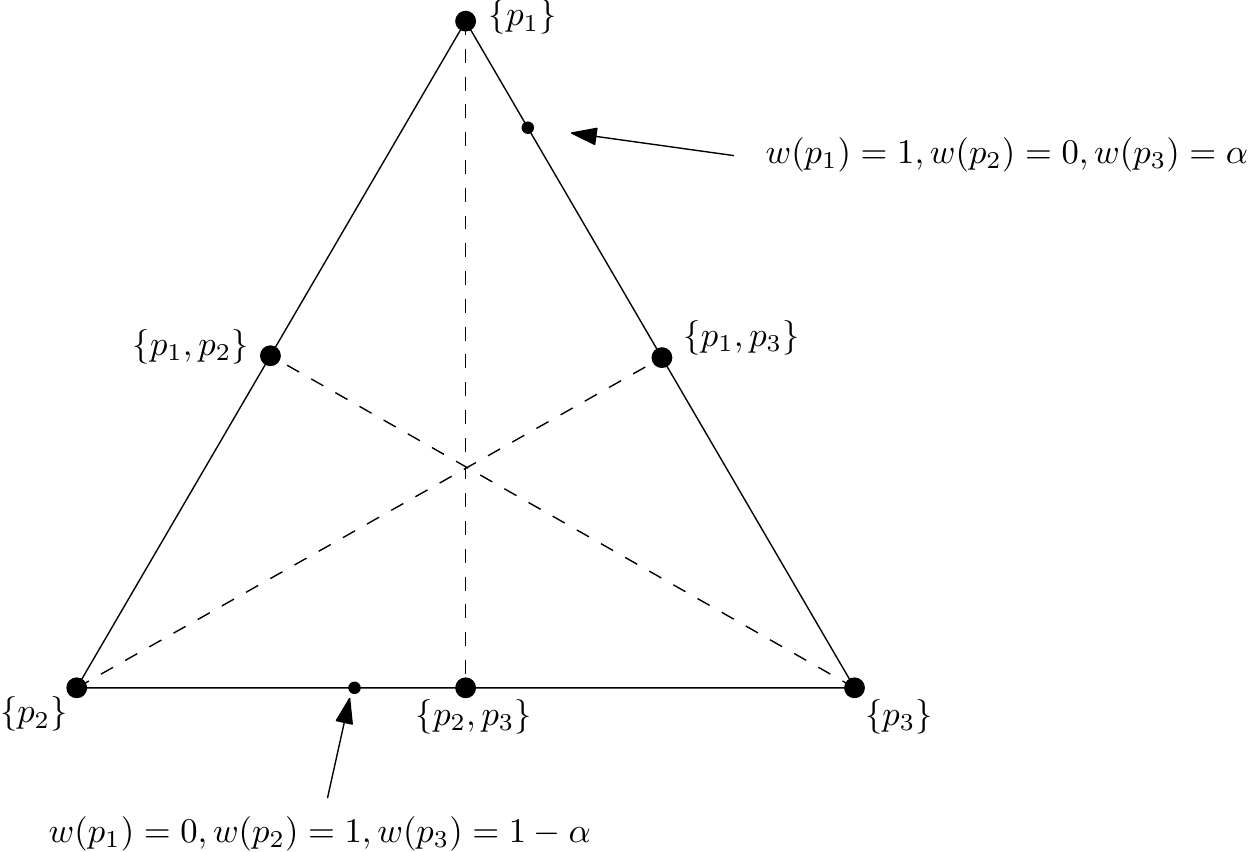}
\caption{Vertices of the barycentric subdivision correspond to strict subsets.}
\label{fig:barycentric}
\end{figure}

\begin{theorem}
Let $\varrho$ be a super-additive depth measure whose depth regions are compact and convex.
Then $\varrho$ is cascading.
\end{theorem}

\begin{proof}
Let $S$ be a weighted point set in $\mathbb{R}^d$.
We may assume that $S$ does not lie in some lower-dimensional affine subspace, otherwise we can do the same arguments in this lower-dimensional ambient space.
We want to show that the cascade condition holds for $S$.
We prove this by induction on $|S|$.
If $|S|< d+1$, then $S$ lies in some lower-dimensional subspace.
If $|S|= d+1$ and $S$ does not lie in some lower-dimensional subspace, then $S$ must be the vertices of a simplex $\Delta$.
By the non-triviality condition (iii), for any point $q$ in a face $F$ of $\Delta$, we have $\varrho(S,q)\geq \min\{w(s):s\text{ is a vertex of } F\}$.
Let $(s_1,\ldots,s_{d+1})$ be an ordering of the points such that $w(s_1)\leq w(s_2)\leq \ldots\leq w(s_{d+1})$ and let $f_i:=s_{d+1}-s_i$ be a basis of $\mathbb{R}^d$.
As $(s_1,\ldots,s_k)$ spans a face of $\Delta$, for each $k\leq d+1$ we have that for $\alpha\leq w(s_k)$ the region $D_\varrho^S(\alpha)$ contains the convex hull of the points $(s_k,\ldots s_{d+1})$.
In particular, its dimension is at least $d+1-k$ and the survival time $\tau_S(f_k)$ of $f_k$ is at least $w(s_k)$.
It follows that $\sum_{i=0}^{d}\tau_S(f_i)\geq w(S)$ and thus by Lemma \ref{lem:integral} $\int_0^{w(S)}t_\alpha d\alpha\geq\sum_{i=0}^d\tau_S(f_i)-w(S)\geq 0$.

Now assume that $|S|\geq d+2$.
By Lemma \ref{lem:cascade_partition}, we can partition $S$ into $S_1$ and $S_2$ whose median regions intersect.
Note that $|S_1|,|S_2|<|S|$, so by the induction hypothesis the cascade condition holds for both $S_1$ and $S_2$.
Thus, by Lemma \ref{lem:cascade_sum}, the cascade condition also holds for $S$.
\end{proof}

As noted above, an example of a super-additive depth measure with compact convex depth regions is Tukey depth.
Thus, we get the following.

\begin{corollary}
Tukey depth is cascading.
\end{corollary}

On the other hand, while Tverberg depth is super-additive, its depth regions are in general not convex; in fact, they are not even connected.
A weak version of Kalai's cascade conjecture claims that the cascade condition holds for the convex hull of Tverberg depth regions.
These depth regions are convex by definition, but the resulting depth measure is in general not super-additive anymore.
So while our approach proves the cascade conjecture for an entire family of depth measures, solving Kalai's cascade conjecture even in its weak form likely requires additional ideas.
As every super-additive depth measure is bounded from below by Tverberg depth, solving the strong version of Kalai's cascade conjecture would imply that all super-additive depth measures are cascading.
Further, as noted by Kalai for Tverberg depth \cite{Kalai_cascade}, any cascading depth measure must enforce deep points.
More precisely, if $\varrho$ is a cascading depth measure and $S$ is a point set in $\mathbb{R}^d$, then there must be a point $q\in\mathbb{R}^d$ for which $\varrho(S,q)\geq\frac{|S|}{d+1}$.
Indeed, if there was no such point, we would have $d_{|S|/(d+1)}=-1$, and even if $d_i=d$ for all $i<\frac{|S|}{d+1}$, the sum $\sum_{i=1}^{|S|}d_i$ would still be negative.
The existence of deep points is the main feature of the next family of depth measures that we study.

\section{A second set of axioms}
\label{sec:axiom2}

The second family of depth measures we consider are \emph{central} depth measures.
A combinatorial depth measure $\varrho: (S^{\mathbb{R}^d},\mathbb{R}^d)\rightarrow \mathbb{R}_{\geq 0}$ is called central if it satisfies the following conditions:
\begin{enumerate}
\item[(i)] for all $S\in S^{\mathbb{R}^d}$ and $q,p\in\mathbb{R}^d$ we have $|\varrho(S,q)-\varrho(S\cup\{p\},q)|\leq 1$ (sensitivity),
\item[(ii)] for all $S\in S^{\mathbb{R}^d}$ and $q\in\mathbb{R}^d$ we have $\varrho(S,q)=0$ for $q\not\in\text{conv}(S)$ (locality),
\item[(iii')] for every $S\in S^{\mathbb{R}^d}$ there is a $q\in\mathbb{R}^d$ for which $\varrho(S,q)\geq\frac{1}{d+1} |S|$ (centrality).
\item[(iv')] for all $S\in S^{\mathbb{R}^d}$ and $q,p\in\mathbb{R}^d$ we have $\varrho(S\cup\{p\},q)\geq \varrho(S,q)$ (monotonicity),
\end{enumerate}

Note that conditions (i) and (ii) are the same as for super-additive depth measures, so by Observation \ref{obs:upper_bound} we have $\varrho(S,q)\leq\text{TD}(S,q)$ for every central depth measure.
Further, for integral depth measures, conditions (iii') and (iv') imply the non-triviality condition (iii) for super-additive depth measures: by the centrality condition (iii') and integrality, every point in the relative interior of a simplex spanned by points of $S$ must have depth at least 1, and thus by the monotonicity condition (iv'), every point in the interior of the convex hull of $S$ has depth at least 1, as by Carath\'{e}odory's theorem each point in the convex hull lies in a simplex.
On the other hand, the super-additivity condition (iv) is stronger than the monotonicity condition (iv'), so at first glance, the families of super-additive depth measures and central depth measures are not comparable.
However, we have seen before that any super-additive depth measure indeed satisfies the centrality condition, so central depth measures are a superset of super-additive depth measures.
It is actually a strict superset, as for example the depth measure whose depth regions are defined as the convex hulls of Tverberg depth regions is central but not super-additive.

While central depth measures enforce deep points by definition, they might still differ a lot locally.
In the following, we will show that for most query points we can bound by how much they differ locally, showing that every central depth measure is a constant factor approximation of Tukey depth.

We say that a query point $q$ is in \emph{general position relative to $S$} if there is no hyperplane containing $q$ and $d$ points of $S$.
Another way to view this is the following: consider again all hyperplanes spanned by $d$ or more points of $S$ and recall that this defines an arrangement $A$ of hyperplanes, whose cells are the connected components of $\mathbb{R}^d\setminus A$. Then a query point $q$ is in general position with respect to $S$ if and only if it lies in a cell of $A$.

\begin{theorem}
\label{thm:central}
Let $\varrho$ be a central depth measure in $\mathbb{R}^d$, $q$ a point in $\mathbb{R}^d$ and $S$ a finite set of points in $\mathbb{R}^d$, where $q$ is in general position with respect to $S$.
Then there exists a constant $c=c(d)$, which depends only on the dimension $d$, such that
$$\text{TD}(S,q)\geq\varrho(S,q)\geq\text{ED}(S,q)-(d+1)\geq c\cdot\text{TD}(S,q)-(d+1).$$
\end{theorem}

Here the first inequality is just Observation \ref{obs:upper_bound}.
As for the second inequality, we would like to argue that if $S$ $k$-encloses $q$ then $\varrho(S,q)\geq k$.
By centrality, there must indeed be a point $q'$ with $\varrho(S,q')\geq k$ (note that $|S|=k(d+1)$ by definition of $k$-enclosing).
As $\varrho$ is bounded from above by Tukey depth, this point has to lie in the \emph{centerpoint region}, that is, the region $D_{\text{TD}}^S(|S|/(d+1))$ of points of Tukey depth at least $|S|/(d+1)$.
However, the point can lie anywhere in the centerpoint region of $S$ and not every point in the centerpoint region is $k$-enclosed by $S$.
We will show that by adding $d+1$ points very close to $q$, we can ensure that $q$ is the only possible centerpoint in the new point set, and the second inequality then follows from sensitivity and monotonicity after removing these points again.
A formal argument for this follows in Lemma \ref{lem:alpha}.

This argument can be generalized even to a relaxation of central depth measures:
We say that a combinatorial depth measure is $\alpha$-central if it satisfies conditions (i), (ii) and (iv'), and the following weak version of condition (iii'):
for every $S\in S^{\mathbb{R}^d}$ there is a $q\in\mathbb{R}^d$ for which $\varrho(S,q)\geq\alpha |S|$ ($\alpha$-centrality).

\begin{lemma}
\label{lem:alpha}
Let $\alpha>\frac{1}{d+2}$, and let $\varrho$ be an $\alpha$-central depth measure.
Let $q$ be a point in $\mathbb{R}^d$ and $S$ a finite set of points in $\mathbb{R}^d$, where $q$ is in general position with respect to $S$.
Then
$$\varrho(S,q)\geq \left(d+2-\frac{1}{\alpha}\right)\cdot\text{ED}(S,q)-(d+1).$$
In particular, if $\alpha=\frac{1}{d+1}$, then $\varrho(S,q)\geq \text{ED}(S,q)-(d+1)$.
\end{lemma}

Before proving Lemma \ref{lem:alpha}, let us state and prove some facts about enclosing sets that will be helpful in the upcoming proof.
For the first fact, recall that the convex hulls of the $S_i$ are \emph{well-separated} if for each $I\subseteq \{1,\ldots,d+1\}$ the convex hulls of $\bigcup_{i\in I}S_i$ and $\bigcup_{i\notin I}S_i$ can be separated by a hyperplane.
This is known to be equivalent to saying that there is no hyperplane that intersects all convex hulls, see e.g.\ Chapter 4.2 in \cite{Handbook}.

\begin{lemma}
\label{lem:wellseparated}
Let $S_1,\ldots, S_{d+1}$ be point sets in $\mathbb{R}^d$ which enclose a point $q$, where $q$ is in general position with respect to $S=S_1\cup\ldots\cup S_{d+1}$.
Then the convex hulls of $S_1,\ldots, S_{d+1}$ are well-separated.
\end{lemma}

\begin{proof}
By the above remark it is enough to show that there is no hyperplane that intersects all the convex hulls $\text{conv}(S_1),\ldots,\text{conv}(S_{d+1})$.
Assume for the sake of contradiction that there is such a hyperplane $H$.
Assume first that $H$ contains $q$.
By the general position assumption, $H$ can contain at most $m\leq d-1$ points of $S$.
In particular the convex hull of $S\cap H$ lies in an affine subspace of dimension at most $m-1\leq d-2$.
If $q$ lies in the convex hull of $S\cap H$, then $S\cap H$ together with any $d-m$ other points of $S$ spans a hyperplane containing $d$ points of $S$ and $q$, which is excluded by the general position assumption.
Thus, $q$ does not lie in the convex hull of $S\cap H$.
Let now $H^+$ denote the closed positive side of $H$.
As $H$ intersects all convex hulls, we have $S_i\cap H^+\neq \emptyset$ for all $i\in\{1,\ldots,d+1\}$.
For each $S_i$ pick some point $s_i\in S_i$.
As all of these points are on the same side of $H$, the point $q$ can only be in the convex hull of the transversal $s_1,\ldots,s_{d+1}$ if it is in the convex hull of the points of the transversal lying on $H$.
However, as $q$ is not in the convex hull of $S\cap H$, we conclude that $q$ is not in the chosen transversal, which is a contradiction to the assumption that $S=(S_1,\ldots,S_{d+1})$ encloses $q$.

Assume now that $H$ does not contain $q$.
Then, analogously, one of the closed sides of $H$ does not contain $q$ but contains a transversal $s_i\in S_i$ which hence does not contain $q$ in its convex hull, which is a contradiction to the fact that $S_1,\ldots, S_{d+1}$ encloses $q$.
\end{proof}

\begin{lemma}
\label{lem:halfspaces}
Let $S_1,\ldots, S_{d+1}$ be point sets in $\mathbb{R}^d$ which enclose a point $q$, where $q$ is in general position with respect to $S=S_1\cup\ldots\cup S_{d+1}$.
Then there are $d+1$ closed halfspaces $H_1,\ldots,H_{d+1}$ such that each $H_i$ contains $q$ on its boundary, $H_i\cap S=S_i$ for each $i$ and $H_1\cup\ldots \cup H_{d+1}=\mathbb{R}^d$.
\end{lemma}

\begin{proof}
By Lemma \ref{lem:wellseparated}, the convex hulls of $S_1,\ldots, S_{d+1}$ are well-separated.
Consider the family $Q_j:=(S_1,\ldots,S_{j-1},S_{j+1},\ldots,S_{d+1})$.
Clearly, $Q_j$ is also well-separated: a hyperplane separating the convex hulls of $\bigcup_{i\in I}S_i$ and $\bigcup_{i\notin I}S_i$ for $I\subseteq \{1,\ldots,d+1\}$ also separates the convex hulls of $\bigcup_{i\in I\setminus\{j\}}S_i$ and $\bigcup_{i\notin I\setminus\{j\}}S_i$.
A generalized version of the Ham-Sandwich theorem states that that for $d$ well-separated point sets $P_1,\ldots,P_d$ for any $(\alpha_1,\ldots, \alpha_d)$, where $0\leq\alpha_i\leq |P_i|$, there is a hyperplane $L$ for whose positive side $L^+$ we have $|L^+\cap P_i|=\alpha_i$ for each $i$ \cite{AlphaHSMasses, AlphaHSPoints}.
We may further assume that $L$ is tangent to each $P_i$, that is, $L$ contains a point of each $P_i$.

Consider now the family $Q_i$ and choose the orientation of hyperplanes intersecting the convex hulls of $Q_i$ in such a way that their positive sides contain $S_i$.
By the generalized Ham-Sandwich theorem there is thus a hyperplane $L_i$ tangent to the sets in $Q_i$ for which $S_i$ lies in the positive side $L_i^+$ and all sets of $Q_i$ lie on the negative side.
We note that $q$ lies in the interior of $L_i^+$: by the general position assumption, it cannot lie in $L_i$, and if it were on the negative side, then the tangent points together with any point from $S_i$ would be a transversal which does not contain $q$ in its convex hull.
Further, note that $L_1^+\cap\ldots\cap L_{d+1}^+\subseteq\text{conv}(S)$ and is thus a bounded convex region, which implies that $L_1^+\cup\ldots\cup L_{d+1}^+=\mathbb{R}^d$.
The result now follows from translating each $L_i$ so that it contains $q$.
\end{proof}

For $S_1,\ldots,S_{d+1}$ enclosing a point $q$, we define for each $S_i$ the cone $C_i$ with apex $q$ by taking positive combinations of vectors $\overrightarrow{qs}$ for $s\in S_i$.

\begin{lemma}
\label{lem:cones}
Let $S_1,\ldots, S_{d+1}$ be point sets in $\mathbb{R}^d$ which enclose a point $q$, where $q$ is in general position with respect to $S=S_1\cup\ldots\cup S_{d+1}$.
Let $p$ be any point in the cone $C_i$.
Then $S_1,\ldots,S_i\cup\{p\},\ldots, S_{d+1}$ also encloses $q$.
\end{lemma}

\begin{proof}
We claim that there is no halfspace $H^+$ with $q$ on its boundary which intersects all of $C_1\setminus\{q\},\ldots,C_{d+1}\setminus\{q\}$.
This follows analogously to Lemma \ref{lem:wellseparated}: if such a halfspace exists, then it contains a transversal which does not contain $q$ in its convex hull.
To show that $S_1,\ldots,S_i\cup\{p\},\ldots, S_{d+1}$ encloses $q$, assume for the sake of contradiction that there is a transversal $s_j\in S_j$ which does not contain $q$ in its convex hull.
As both the convex hull and $q$ are convex and compact, by the separation theorem there exists a hyperplane $H$ separating the two.
Translate the separating hyperplane $H$ such that it contains $q$ and orient it in such a way that its positive side $H^+$ contains the convex hull of the transversal.
Then $H^+$ intersects all of $C_1\setminus\{q\},\ldots,C_{d+1}\setminus\{q\}$, which is a contradiction.
\end{proof}

We are now ready to prove Lemma \ref{lem:alpha}:

\begin{proof}
Let $\text{ED}(S,q)=k$ and let $S'=(S'_1,\ldots,S'_{d+1})$ be a witness subset.
Recall that by monotonicity, we have $\varrho(S,q)\geq\varrho(S',q)$.
We first claim that $\text{TD}(S',q')\leq k$ for all $q'\in\mathbb{R}^{d}$: indeed, by Lemma \ref{lem:halfspaces}, one of the halfspaces $H_1,\ldots,H_{d+1}$ contains $q'$ and exactly one of the $S'_i$.
On the other hand, as any transversal of $S'_1,\ldots,S'_{d+1}$ gives a simplex having the point $q$ in its interior, we get that $\text{TvD}(S',q)=k$, and thus, as $\text{TD}(S',q)\geq\text{TvD}(S',q)$ and $\text{TD}(S',q)\leq k$ we get that $\text{TD}(S',q)= k$.

Let $\alpha':=(d+1)\alpha$ and let $m:=\lfloor\frac{1-\alpha'}{\alpha'}k+1\rfloor$.
Let $Z$ be the cell of the hyperplane arrangement induced by $S'$ which contains $q$ and let $B$ be a small enough neighborhood of $q$ such that $B$ lies completely in the interior of $Z$.
For each $i$ add $m$ points in $C_i\cap B$.
By Lemma \ref{lem:cones}, this new point set $P$ $(k+m)$-encloses $q$.
The new point set $P$ has $(d+1)(k+m)$ many points, and we have
$$\alpha|P|=\alpha' (k+m)> \alpha' \left(k+\frac{1-\alpha'}{\alpha'}k\right)=\alpha' k+(1-\alpha' )k=k.$$

We claim that for any point $q'$ that is not in $Z$ we still have $\text{TD}(P,q')\leq k$: by Lemma \ref{lem:halfspaces}, the point $q'$ lies in one of the halfspaces $H_1,\ldots,H_{d+1}$, without loss of generality $H_1$. Translate $H_1$ until it does not intersect $B$ anymore. As $B$ was chosen sufficiently small, $H_1$ still contains $q'$ and we have $|P\cap H_1|=|S'_1|=k$ by construction.

Thus, as $\varrho(P,q')\leq\text{TD}(P,q')$, the only points $q'$ for which $\varrho(P,q')\geq\alpha|P|$ is possible are by construction in $B$.
As $B\subset Z$ and all points in $Z$ had the same depth before adding the new points, we can assume that we have $\varrho(P,q)\geq\alpha|P|$.
By sensitivity we now have

\begin{align}
\varrho(S',q)\geq\varrho(P,q)-(d+1)m \\
\geq \alpha' (k+m)-(d+1)m \\
= \alpha' k-(d+1-\alpha')m \\
\geq \alpha' k-(d+1-\alpha')\left(\frac{1-\alpha'}{\alpha'}k+1\right) \\
= \alpha' k-\frac{(d+1-\alpha')(1-\alpha')}{\alpha'}k-(d+1)+\alpha' \\
\geq (\alpha'^2-(d+1)+\alpha'+(d+1)\alpha'-\alpha'^2)\frac{k}{\alpha'}-(d+1) \\
= \frac{(d+2)\alpha'-(d+1)}{\alpha'}k-(d+1).
\end{align}

Plugging in $\alpha':=(d+1)\alpha$ we get
$$\varrho(S,q)\geq\frac{(d+2)(d+1)\alpha-(d+1)}{(d+1)\alpha}k-(d+1)=\left(d+2-\frac{1}{\alpha}\right)k-(d+1).$$
As $(d+2-\frac{1}{\alpha})>0$ for $\alpha>\frac{1}{d+2}$ and $(d+2-\frac{1}{\alpha})=1 $, for $\alpha=\frac{1}{d+1}$ the claim follows.
\end{proof}

The most involved part of Theorem \ref{thm:central} is the last inequality, which we will prove in the next section.

\section{A lower bound for enclosing depth}
\label{sec:enclosing}

In this section, we will give a new proof of a lower bound on the enclosing depth in terms of Tukey depth:

\begin{theorem}[$E(d)$]
There is a constant $c_1=c_1(d)$ such that for all $S\in S^{\mathbb{R}^d}$ and $q\in\mathbb{R}^d$ we have $\text{ED}(S,q)\geq c_1\cdot\text{TD}(S,q)$.
\end{theorem}

We will denote this statement in dimension $d$ by $E(d)$.
Note that $E(1)$ is true and $c_1(1)=1$.

This theorem has appeared in the literature before, both implicitly and explicitly.
First, it could be proved using the semi-algebraic same type lemma due to Fox, Pach and Suk \cite{semialgebraic}, combined with the first selection lemma (see e.g.~\cite{MatousekDiscreteGeometry}).
Alternatively, a result by Pach \cite{Pach} gives a lower bound on a colorful variant of enclosing depth in terms of colorful simplicial depth. As colorful simplicial depth can be bounded from below in terms of simplicial depth, which in turn can be bounded from below in terms of Tukey depth (see e.g.~\cite{Uli_Phd}), the result follows.
An explicit lower bound on enclosing depth in terms of Tukey depth was given by Fabila-Monroy and Huemer \cite{Fabila}.

Here we will give a different proof for two reasons: first, the bounds on $c_1$ that our proof gives are better than the bounds we get from the other proofs.
Second, our proof shows an intimate relation of enclosing depth to the following positive fraction Radon theorem on certain bichromatic point sets, which we believe to be of independent interest.

Let $P=R\cup B$ be a bichromatic point set with color classes $R$ (red) and $B$ (blue).
We say that $B$ \emph{surrounds} $R$ if for every halfspace $h$ we have $|B\cap h|\geq |R\cap h|$.
Note that this in particular implies $|B|\geq|R|$.
The positive fraction Radon theorem is now the following:

\begin{theorem}[$R(d)$]
Let $P=R\cup B$ be a bichromatic point set where $B$ surrounds $R$.
Then there is a constant $c_2=c_2(d)$ such that there are integers $a$ and $b$ and pairwise disjoint subsets $R_1,\ldots,R_a\subseteq R$ and $B_1,\ldots,B_b\subseteq B$ with
\begin{enumerate}
\item $a+b=d+2$,
\item $|R_i|\geq c_2\cdot |R|$ for all $1\leq i\leq a$,
\item $|B_i|\geq c_2\cdot |R|$ for all $1\leq i\leq b$,
\item for every transversal $r_1\in R_1,\ldots,r_a\in R_a, b_1\in B_1,\ldots, b_b\in B_b$, we have $\text{conv}(r_1,\ldots,r_a)\cap \text{conv}(b_1,\ldots,b_b)\neq\emptyset$.
\end{enumerate}
\end{theorem}

In other words, the Radon partition respects the color classes.
We will denote the above statement in dimension $d$ by $R(d)$.

\begin{lemma}
$R(1)$ can be satisfied choosing $a=1$, $b=2$ and $c_2(1)=\frac{1}{3}$.
\end{lemma}

\begin{proof}
Consider two points $x_1$ and $x_2$ such that there are exactly $\frac{|R|}{3}$ blue points to the left of $x_1$ and to the right of $x_2$, respectively.
Define $B_1$ as the set of blue points left of $x_1$ and $B_2$ as the set of blue points right $x_2$.
We then have $|B_1|=|B_2|=\frac{1}{3}|R|$.
Further, as $B$ surrounds $R$, we have at most $\frac{|R|}{3}$ red points to the left of $x_1$, and also to the right of $x_2$.
In particular, there are at least $\frac{|R|}{3}$ red points between $x_1$ and $x_2$.
Let now $R_1$ be any subset of $\frac{|R|}{3}$ red points between $x_1$ and $x_2$.
It follows from the construction that $\text{conv}(R_1)\cap \text{conv}(B_1,B_2)\neq\emptyset$.
\end{proof}

In the following, we will prove that $R(d-1)\Rightarrow E(d)$ and that $E(d-1)\Rightarrow R(d)$.
By induction, these two claims then imply the above theorems.

\begin{lemma}
$R(d-1)\Rightarrow E(d)$.
\end{lemma}

\begin{proof}
We assume that $q$ does not coincide with a point from $S$, otherwise we just remove that point from $S$.
Assume that $\text{TD}(S,q)=k$ and let $h$ be a witnessing hyperplane which contains $q$ but no points of $S$.
Without loss of generality, assume that $q$ is the origin and that $h$ is the hyperplane through the equator on $S^{d-1}\subseteq \mathbb{R}^d$, with exactly $k$ points below.
Color the points below $h$ red and the points above $h$ blue.
Now, for every point $p\in S$, consider the line through $p$ and $q$ and let $p'$ be the intersection of that line with the tangent hyperplane to the north pole of $S^{d-1}$.
Color $p'$ the same color as $p$.
This gives a bichromatic point set $S'=R\cup B$ in $\mathbb{R}^{d-1}$.
Further, in $S'$, we have that $B$ surrounds $R$: Assume there is a hyperplane $\ell$ (in $\mathbb{R}^{d-1}$) with $r$ red points and $b$ blue points on its positive side, where $r>b$.
In $\mathbb{R}^d$, this lifts to a hyperplane containing $q$ with $k-r$ red points and $b$ blue points on its positive side (note that there are exactly $k$ red points).
However, $k-r+b<k$, whenever $r>b$, thus we would have $\text{TD}(S,q)<k$, which is a contradiction.

As we now have a point set in $\mathbb{R}^{d-1}$, in which $B$ surrounds $R$, we can apply $R(d-1)$ to find families of $d+2$ subsets of $S'$, each of size $c_2\cdot k$, some red and some blue, such that in each transversal the color classes form a Radon partition.
We claim that the corresponding subsets of $S$ $c_2\cdot k$-enclose $q$.
Pick some transversal (which we call the original red and blue points) and consider the corresponding subset in $S'$.
Let $z$ be a point in the intersection of the convex hulls of the two color classes, and let $g$ be the line through $z$ and $q$.
As $z$ is in the convex hull of the blue points, there is a point $z^+$ on $g$ which is in the convex hull of the original blue points, and thus above $h$.
Similarly, there is a point $z^-$ on $g$ which is in the convex hull of the original red points, and thus below $h$.
As $q$ is in the convex hull of $z^+$ and $z^-$, it is thus in the convex hull of the original blue and red points.
\end{proof}

In particular, this proof shows that $c_1(d)\geq c_2(d-1)$.

For the proof of the second implication, we need to recall a few results, starting with the \emph{Same Type Lemma} by B\'{a}r\'{a}ny and Valtr \cite{Barany}.

\begin{theorem}[Theorem 2 in \cite{Barany}]
\label{lem:selection}
For every two natural numbers $d$ and $m$ there is a constant $c_3(d,m)>0$ with the following property: Given point sets $X_1,\ldots,X_m\subseteq\mathbb{R}^d$ such that $X_1\cup\ldots\cup X_m$ is in general position, there are subsets $Y_i\subseteq X_i$ with $|Y_i|\geq c_3\cdot|X_i|$ such that all transversals of the $Y_i$ have the same order type.
\end{theorem}

We note that while Theorem 2 in \cite{Barany} is stated using the general position assumption, in Remark 5 of the same paper it is mentioned that the result still holds without it and, in fact, even holds for Borel measures instead of point sets.

From the proof in \cite{Barany}, we get $c_3(d,m)=2^{-m^{O(d)}}$. This bound has been improved in \cite{semialgebraic} to $c_3(d,m)=2^{-O(d^3m\log m)}$.

The second result that we will need is the \emph{Center Transversal Theorem}, proved independently by Dol'nikov \cite{Dolnikov} as well as Zivaljevi{\'c} and Vre{\'c}ica \cite{Zivaljevic}.
We will only need the version for two colors, so we state it in this restricted version:

\begin{theorem}[Center Transversal for two colors]
Let $\mu_1$ and $\mu_2$ be two finite Borel measures on $\mathbb{R}^d$.
Then there exists a line $\ell$ such that for every closed halfspace $H$ which contains $\ell$ and every $i\in\{1,2\}$ we have $\mu_i(H)\geq\frac{\mu_i(\mathbb{R}^d)}{d}$.
\end{theorem}

Such a line $\ell$ is called a \emph{center transversal}.
By a standard argument (replacing points with balls of small radius, see e.g. \cite{Matousek}), the same result also holds for two point sets $P_1, P_2$, where $\mu_i(H)$ is replaced by $|P_i\cap H|$.
As we will need similar ideas later, we will briefly sketch a proof of the above Theorem.
Consider some $(d-1)$-dimensional linear subspace $F$, i.e., a hyperplane through the origin, and project both measures to it.
For each projected measure, consider the centerpoint region (i.e., the region of Tukey depth $\geq\frac{\mu_i(\mathbb{R}^d)}{(d-1)+1}$).
This is a non-empty, convex set, so it has a unique center of mass, which we will denote by $g_i(F)$.
Rotating the subspace $F$ in continuous fashion, these centers of mass also move continuously, so the $g_i(F)$ are two continuous assignments of points to the set of all $(d-1)$-dimensional linear subspaces.
The result then follows from the following Lemma, again proved independently by Dol'nikov (\cite{Dolnikov}, Lemma 1) as well as Zivaljevi{\'c} and Vre{\'c}ica (\cite{Zivaljevic}, Proposition 2).
In their works, the result is phrased in terms of sections of the canonical bundle over a the Grassmannian manifold, we use a rephrased version which is less general than their statements.

\begin{lemma}
\label{lem:sections}
Let $g_1$ and $g_2$ be two continuous assignments of points to the set of all $(d-1)$-dimensional linear subspaces of $\mathbb{R}^d$.
Then there exists such a subspace $F$ in which $g_1(F)=g_2(F)$.
\end{lemma}

Note that in order to apply this Lemma, we had to choose in a continuous way a centerpoint.
If the two measures can be separated by a hyperplane, we can do something similar with the center transversal:

\begin{lemma}
\label{lem:unique_center_transversal}
Let $\mu_1$ and $\mu_2$ be two finite Borel measures on $\mathbb{R}^d$, which can be separated by a hyperplane.
Then there is a canonical choice of a center transversal, that is, a choice of a center transversal which is continuous under continuous changes of the measures.
\end{lemma}

\begin{proof}
Let $x_1,\ldots,x_d$ be the basis vectors of $\mathbb{R}^d$ and assume without loss of generality that the hyperplane $H: x_d=0$ separates the two measures $\mu_1,\mu_2$, with $\mu_1$ being above $H$ and $\mu_2$ below.
For any $(d-1)$-dimensional linear subspace $F$, consider the orthogonal projection $\pi_F: \mathbb{R}^d\rightarrow F$.
Note that if $F$ is orthogonal to $H$, then $\pi_F(H)$ separates $\pi_F(\mu_1)$ and $\pi_F(\mu_2)$, so there is no center transversal parallel to $H$.
It thus suffices to consider only (oriented) subspaces which point upwards (in the sense that the $x_d$-component in their normal vector is $>0$).
The space of these subspaces is homeomorphic to the upper hemisphere $S^+$ of $S^{d-1}$.

For each such subspace $F$ let $G_i(F)$ denote the centerpoint region of the projected mass $\pi_F(\mu_i)$ and let $C\subseteq S^+$ be the subset of $F$ where $G_1(F)$ and $G_2(F)$ intersect.
From the center transversal theorem, we know that $C$ is not empty.
Also note that $C$ is independent of the choice of $H$.
We claim that $C$ is a convex subset of $S^+$.

Consider two subspaces $F_1$ and $F_2$ with $G_1(F_1)\cap G_2(F_1)\neq\emptyset$ and $G_1(F_2)\cap G_2(F_2)\neq\emptyset$.
The shortest path between $F_1$ and $F_2$ corresponds to a rotation around a $(d-2)$-dimensional axis.
Let $F_3$ be a subspace along this rotation, and assume for the sake of contradiction that $G_1(F_3)\cap G_2(F_3)=\emptyset$.
This means that there is a hyperplane $\ell$ in $F_3$ separating the two centerpoint regions $G_1(F_3)$ and $G_2(F_3)$.
In particular, by the definition of the centerpoint regions, $\ell$ has less than a $\frac{1}{d}$-fraction of $\pi_{F_3}(\mu_1)$ on its positive side, and less than a $\frac{1}{d}$-fraction of $\pi_{F_3}(\mu_2)$ on its negative side, or vice versa.
Assume without loss of generality that $\ell$ goes through the origin.
Consider a point in the support of one of the measures.
During the rotation of $F_3$ the projection of this point moves along a line in the projection.
In fact, all points move along parallel lines, and the points in the support of $\mu_1$ move in the opposite direction of the points in the support of $\mu_2$ and each point crosses a fixed hyperplane through the origin in the projection at most once.
Thus, without loss of generality, the fraction of $\pi_F(\mu_1)$ on the positive side of $\ell$ only decreases during the rotation from $F_1$ to $F_2$, and the same holds for the fraction of $\pi_F(\mu_2)$ on the negative side of $\ell$.
But then, $\ell$ still separates $G_1(F_2)$ and $G_2(F_2)$ in $F_2$, and thus $F_2\notin C$, which is a contradiction.
This shows that $C$ is indeed a convex subset of $S^+$.

Let now $F$ be the center of mass of $C$.
In $F$, we have that $G:=G_1(F)\cap G_2(F)$ is a non-empty convex set.
Let $g$ be the center of mass of $G$.
As all the above maps are continuous, the preimage $\pi_F^{-1}(g)$ is now a canonical center transversal.

\end{proof}

Again, the same statement holds for point sets.
With these tools at hand, we are now ready to prove the second part of the induction.

\begin{lemma}
$E(d-1)\Rightarrow R(d)$.
\end{lemma}

\begin{proof}
Let $P=R\cup B$ be a bichromatic point set where $B$ surrounds $R$.
Recall that we want to find pairwise disjoint subsets $R_1,\ldots,R_a\subseteq R$ and $B_1,\ldots,B_b\subseteq B$ with $a+b=d+2$, all of which contain at least $c_2\cdot |R|$ points and with the property that for every transversal $r_1\in R_1,\ldots,r_a\in R_a, b_1\in B_1,\ldots, b_b\in B_b$, we have $\text{conv}(r_1,\ldots,r_a)\cap \text{conv}(b_1,\ldots,b_b)\neq\emptyset$.

Let $\ell$ be a line through the origin.
Sweep a hyperplane orthogonal to $\ell$  from one side to the other (without loss of generality from left to right).
Let $h_1$ be a sweep hyperplane with exactly $\frac{|R|}{3}$ blue points to the left, and let $A_1$ be the set of these blue points.
Similarly, let $A_2$ be a set of exactly $\frac{|R|}{3}$ blue points to the right of a sweep hyperplane $h_2$.
Let $c$ be the canonical center transversal of $A_1$ and $A_2$ given by Lemma \ref{lem:unique_center_transversal} and let $g$ be the $(d-1)$-dimensional linear subspace which is orthogonal to $c$.
Note that it follows from the proof of Lemma \ref{lem:unique_center_transversal} that $g$ cannot be orthogonal to the sweep hyperplanes.
We denote the projection of $c$ to $g$ as $c_A$.
Note that $c_A$ is a centerpoint of the projections of $A_1$ and of $A_2$ to $g$.
Now, consider the set $M$ of all red points between $h_1$ and $h_2$ and note that as the blue points surround the red points we have $|M|\geq\frac{|R|}{3}$.
Project $M$ to $g$ and denote by $c_M$ the center of mass of the centerpoint region of the projected point set.
We claim that there exists a choice of a line $\ell$, such that $c_M=c_A$.
As $g$ is not orthogonal to a sweep hyperplane, there is a unique shortest rotation which rotates $g$ to a hyperplane orthogonal to $\ell$.
In particular, these rotations give a homeomorphism from the space of hyperplanes $g$ to the space of hyperplanes orthogonal to the $\ell$'s, which is the space of all $(d-1)$-dimensional linear subspaces.
Further, as rotations and projections are continuous and $c$ is a canonical center transversal, $c_A$ and $c_M$ are continuous assignments of points, thus the above claim follows from Lemma \ref{lem:sections}.

So assume now that $c_M=c_A$.
In particular, $c$ is a center transversal for $A_1$, $A_2$ and $M$.
Project $A_1$ to $g$.
The projection of $c$ is a centerpoint of the projection of $A_1$ in $g$ and $g$ has dimension $d-1$, thus by the statement $E(d-1)$ there are three subsets $A_{1,1}, \ldots A_{1,d}$ of $A_1$, each of size $c_1\cdot |A_1|$ whose projections enclose the projection of $c$.
The analogous arguments gives subsets $A_{2,1}, \ldots, A_{2,d}$ of $A_2$ and $M_1,\ldots,M_d$ of $M$.
Consider now these $3d$ subsets.
By Theorem \ref{lem:selection} there are subsets $A'_{1,1},\ldots,M'_d$, each of size linear in the size of the original subset, such that each transversal of the subsets has the same order type.
Consider such a transversal.
By construction, the $d$ points of $A_1$ contain in their convex hull a point on $c$ which is to the left of $h_1$.
Similarly, the $d$ points of $A_2$ contain in their convex hull a point on $c$ to the right of $h_2$.
Finally, the $d$ points of $M$ contain in their convex hull a point on $c$ between $h_1$ and $h_2$.
Thus, the convex hulls of the blue points (from $A_1$ and $A_2$) and the red points (from $M$) intersect.
In particular, by Kirchberger's theorem \cite{Kirchberger}, there is a subset of $d+2$ red and blue points, which form a Radon partition where the convex hull of the red points intersects the convex hull of the blue points.
Now, choose the subsets from which these points were selected.
As every transversal of these subsets has the same order type, every transversal gives a Radon partition which respects the color classes, thus these subsets satisfy the required properties.
\end{proof}

This proof shows that $c_2(d)\geq\frac{c_3(d,3d)}{3}c_1(d-1)$.
Using the bound on $c_3$ from \cite{semialgebraic} and $c_1(d)\geq c_2(d-1)$, we thus get $c_2(d)=\Omega(\frac{c_2(d-2)}{3\cdot 2^{d^4\log d}})=\ldots=\Omega(\frac{1}{3^{d/2}\cdot 2^{d^5\log d}})$, and as $c_1(d)\geq c_2(d-1)$ we get the same asymptotics for $c_1$.
For comparison, for the constant $c'$ proven by Fabila-Monroy and Huemer, it follows from their proof that $c'(d)\geq\frac{c_3(d,2d)}{2}c'(d-1)$, which gives 
$c'(d)\in\Omega(\frac{1}{2^{d}\cdot 2^{d^5\log d}})$.

Combining this with the results from Section \ref{sec:axiom2}, we get that any central depth measure is an approximation of Tukey depth.
In fact, by Lemma \ref{lem:alpha} this even holds for many $\alpha$-central depth measures.

\begin{corollary}
Let $\varrho$ be an $\alpha$-central depth measure on $\mathbb{R}^d$ where $\alpha>\frac{1}{d+2}$.
Then there exists a constant $c=c(d)$ such that for every point set $S$ and query point $q$ in $\mathbb{R}^d$ we have
$$\text{TD}(S,q)\geq\varrho(S,q)\geq c\cdot \text{TD}(S,q).$$
\end{corollary}

\section{Conclusion}

We have introduced two families of depth measures, called super-additive depth measures and central depth measures, where the first is a strict subset of the second.
We have shown that all these depth measures are a constant-factor approximation of Tukey depth.

It is known that Tukey depth is coNP-hard to compute when both $|S|$ and $d$ is part of the input \cite{Tukey_hard}, and it is even hard to approximate \cite{Tukey_apx} (see also \cite{Wagner}).
Our result is thus an indication that central depth measures are hard to compute.
However, this does not follow directly, as our constant has a doubly exponential dependence on $d$.
It is an interesting open problem whether the approximation factor can be improved.

Further, we have introduced a new depth measure called enclosing depth, which is neither super-additive nor central, but still is a constant-factor approximation of Tukey depth.
As it turns out, this depth measure is intimately related to a constant fraction Radon theorem on bi-colored point sets.
Finally, we have shown that any super-additive depth measure whose depth regions are convex is cascading.

This last result is motivated by Kalai's cascade conjecture, which, in the terminology of this paper, states that Tverberg depth is cascading.
While this conjecture remains open, we hope that our results might be useful for an eventual proof.

There is a depth measure which has attracted a lot of research, which does not fit into our framework: simplicial depth ($\text{SD}$).
The reason for this is that while the depths studied in this paper are linear in the size of the point set, simplicial depth has values of size $O(|S|^{d+1})$.
However, after the right normalization, simplicial depth can be reformulated to satisfy all conditions except super-additivity and centrality.
It would be interesting to see whether there is some function $g$ depending on point sets and query points such that the depth measure $\frac{\text{SD}(S,q)}{g(S,q)}$ is super-additive.
Such a function, if it exists, could potentially be used to improve bounds for the first selection lemma (see e.g.~\cite{MatousekDiscreteGeometry}).

\bibliographystyle{plainurl} %
\bibliography{refs}

\begin{thebibliography}{10}

\bibitem{aloupis}
Greg Aloupis.
\newblock Geometric measures of data depth.
\newblock In {\em Data Depth: Robust Multivariate Analysis, Computational
  Geometry and Applications}, pages 147--158, 2003.

\bibitem{Tukey_apx}
Edoardo Amaldi and Viggo Kann.
\newblock The complexity and approximability of finding maximum feasible
  subsystems of linear relations.
\newblock {\em Theoretical Computer Science}, 147(1):181 -- 210, 1995.

\bibitem{Avis}
David Avis.
\newblock The m-core properly contains the m-divisible points in space.
\newblock {\em Pattern recognition letters}, 14(9):703--705, 1993.

\bibitem{AlphaHSMasses}
Imre B{\'a}r{\'a}ny, Alfredo Hubard, and Jes{\'u}s Jer{\'o}nimo.
\newblock Slicing convex sets and measures by a hyperplane.
\newblock {\em Discrete \& Computational Geometry}, 39(1):67--75, 2008.

\bibitem{BaranyTverberg}
Imre B{\'a}r{\'a}ny and Pablo Sober{\'o}n.
\newblock Tverberg’s theorem is 50 years old: a survey.
\newblock {\em Bulletin of the American Mathematical Society}, 55(4):459--492,
  2018.

\bibitem{Barany}
Imre B{\'a}r{\'a}ny and Pavel Valtr.
\newblock A positive fraction {E}rd\"{o}s-{S}zekeres theorem.
\newblock {\em Discrete \& Computational Geometry}, 19(3):335--342, 1998.

\bibitem{Bourgin}
DG~Bourgin.
\newblock On some separation and mapping theorems.
\newblock {\em Commentarii Mathematici Helvetici}, 29(1):199--214, 1955.

\bibitem{Wagner}
Dan Chen, Pat Morin, and Uli Wagner.
\newblock Absolute approximation of {T}ukey depth: {T}heory and experiments.
\newblock {\em Computational Geometry}, 46(5):566 -- 573, 2013.
\newblock Geometry and Optimization.

\bibitem{Dolnikov}
VL~Dol'nikov.
\newblock Transversals of families of sets in in $\mathbb{R}^n$ and a
  connection between the {H}elly and {B}orsuk theorems.
\newblock {\em Russian Academy of Sciences. Sbornik Mathematics}, 79(1):93,
  1994.

\bibitem{Fabila}
Ruy Fabila-Monroy and Clemens Huemer.
\newblock Caratheodory’s theorem in depth.
\newblock {\em Discrete \& Computational Geometry}, 58(1):51--66, 2017.

\bibitem{semialgebraic}
Jacob Fox, J{\'a}nos Pach, and Andrew Suk.
\newblock A polynomial regularity lemma for semialgebraic hypergraphs and its
  applications in geometry and property testing.
\newblock {\em SIAM Journal on Computing}, 45(6):2199--2223, 2016.

\bibitem{Sariel}
Sariel Har-Peled and Timothy Zhou.
\newblock {Improved Approximation Algorithms for Tverberg Partitions}.
\newblock {\em arXiv preprint arXiv:2007.08717}, 2020.

\bibitem{hodges}
Joseph~L. Hodges.
\newblock A bivariate sign test.
\newblock {\em The Annals of Mathematical Statistics}, 26(3):523--527, 1955.

\bibitem{Tukey_hard}
D.S. Johnson and F.P. Preparata.
\newblock The densest hemisphere problem.
\newblock {\em Theoretical Computer Science}, 6(1):93 -- 107, 1978.

\bibitem{Kalai_cascade}
Gil Kalai.
\newblock Combinatorics with a geometric flavor.
\newblock {\em Visions in Mathematics: GAFA 2000 Special Volume, Part II}, page
  742, 2011.

\bibitem{KalaiTalk}
Gil Kalai.
\newblock {Problems in Geometric and Topological Combinatorics}.
\newblock Lecture at FU Berlin, 2011.

\bibitem{KalaiBirthday}
Gil Kalai.
\newblock {Problems for Imre B{\'a}r{\'a}ny’s Birthday}.
\newblock {\em Discrete Geometry and Convexity in Honour of Imre
  B{\'a}r{\'a}ny}, page~59, 2017.

\bibitem{Kirchberger}
Paul Kirchberger.
\newblock {{\"U}ber Tchebychefsche Ann{\"a}herungsmethoden}.
\newblock {\em Mathematische Annalen}, 57(4):509--540, 1903.

\bibitem{LiuSimplicial}
Regina~Y. Liu.
\newblock On a notion of data depth based on random simplices.
\newblock {\em The Annals of Statistics}, 18(1):405--414, 1990.

\bibitem{Liu}
Regina~Y. Liu, Jesse~M. Parelius, and Kesar Singh.
\newblock Multivariate analysis by data depth: descriptive statistics, graphics
  and inference.
\newblock {\em Ann. Statist.}, 27(3):783--858, 06 1999.

\bibitem{MatousekDiscreteGeometry}
Ji\v{r}\'{i} Matou{\v{s}}ek.
\newblock {\em Lectures on discrete geometry}, volume 212 of {\em Graduate
  texts in mathematics}.
\newblock Springer, 2002.

\bibitem{Matousek}
Ji\v{r}\'{i} Matou{\v{s}}ek.
\newblock {\em {Using the Borsuk-Ulam Theorem: Lectures on Topological Methods
  in Combinatorics and Geometry}}.
\newblock Springer Publishing Company, Incorporated, 2007.

\bibitem{flags1}
Pierre-Lo{\"\i}c M{\'e}liot.
\newblock {\em Representation theory of symmetric groups}.
\newblock Chapman and Hall/CRC, 2017.

\bibitem{Miller}
Kim Miller, Suneeta Ramaswami, Peter Rousseeuw, J.~Antoni Sellar{\`e}s, Diane
  Souvaine, Ileana Streinu, and Anja Struyf.
\newblock Efficient computation of location depth contours by methods of
  computational geometry.
\newblock {\em Statistics and Computing}, 13(2):153--162, 2003.

\bibitem{Mosler}
Karl Mosler.
\newblock {\em Depth Statistics}, pages 17--34.
\newblock Springer Berlin Heidelberg, Berlin, Heidelberg, 2013.

\bibitem{Pach}
J{\'a}nos Pach.
\newblock {A Tverberg-type result on multicolored simplices}.
\newblock {\em Computational Geometry}, 10(2):71--76, 1998.

\bibitem{CP}
Richard Rado.
\newblock A theorem on general measure.
\newblock {\em Journal of the London Mathematical Society}, 21:291--300, 1947.

\bibitem{Reay}
John~R. Reay.
\newblock {Several generalizations of Tverberg's theorem}.
\newblock {\em Israel Journal of Mathematics}, 34(3):238--244, 1979.

\bibitem{PabloAlgo}
David Rolnick and Pablo Sober{\'o}n.
\newblock {Algorithms for Tverberg's theorem via centerpoint theorems}.
\newblock {\em arXiv preprint arXiv:1601.03083}, 2016.

\bibitem{Roudneff1}
Jean-Pierre Roudneff.
\newblock {Partitions of Points into Simplices with k-dimensional Intersection.
  Part I: The Conic Tverberg’s Theorem}.
\newblock {\em European Journal of Combinatorics}, 22(5):733--743, 2001.

\bibitem{Roudneff2}
Jean-Pierre Roudneff.
\newblock {Partitions of Points into Simplices with k-dimensional Intersection.
  Part II: Proof of Reay’s Conjecture in Dimensions 4 and 5}.
\newblock {\em European Journal of Combinatorics}, 22(5):745--765, 2001.

\bibitem{Roudneff3}
Jean-Pierre Roudneff.
\newblock {New cases of Reay’s conjecture on partitions of points into
  simplices with k-dimensional intersection}.
\newblock {\em European Journal of Combinatorics}, 30(8):1919--1943, 2009.

\bibitem{AlphaHSPoints}
William Steiger and Jihui Zhao.
\newblock Generalized ham-sandwich cuts.
\newblock {\em Discrete \& Computational Geometry}, 44(3):535--545, 2010.

\bibitem{Handbook}
Csaba~D. Toth, Joseph O'Rourke, and Jacob~E. Goodman.
\newblock {\em Handbook of discrete and computational geometry}.
\newblock Chapman and Hall/CRC, 2017.

\bibitem{tukey}
John~W. Tukey.
\newblock Mathematics and the picturing of data.
\newblock In {\em Proc. International Congress of Mathematicians}, pages
  523--531, 1975.

\bibitem{tverberg}
Helge Tverberg.
\newblock A generalization of {R}adon's theorem.
\newblock {\em Journal of the London Mathematical Society}, 1(1):123--128,
  1966.

\bibitem{Uli_Phd}
Ulrich Wagner.
\newblock On k-sets and applications.
\newblock {\em ETH Z{\"u}rich, Z{\"u}rich}, 2003.

\bibitem{Yang}
Chung-Tao Yang.
\newblock On theorems of {B}orsuk-{U}lam, {K}akutani-{Y}amabe-{Y}ujob{\^o} and
  {D}yson, {I}.
\newblock {\em Annals of Mathematics}, pages 262--282, 1954.

\bibitem{Zivaljevic}
Rade~T. Zivaljevi{\'c} and Sini{\v{s}}a~T. Vre{\'c}ica.
\newblock An extension of the ham sandwich theorem.
\newblock {\em Bulletin of the London Mathematical Society}, 22(2):183--186,
  1990.

\bibitem{Zuo}
Yijun Zuo and Robert Serfling.
\newblock Structural properties and convergence results for contours of sample
  statistical depth functions.
\newblock {\em Ann. Statist.}, 28(2):483--499, 04 2000.

\end{thebibliography}

\end{document}